\documentclass[12pt]{amsart}

\usepackage{amsmath,amsthm,amssymb,amssymb, paralist, xspace, graphicx, url, amscd, euscript, mathrsfs,stmaryrd,epic,eepic,color}

\usepackage[all]{xy}
\usepackage{amsmath,amscd}
\SelectTips{cm}{}

\usepackage[colorlinks=true]{hyperref}

\usepackage{pstricks}

\numberwithin{equation}{subsection}

1

\numberwithin{subsection}{section}

\allowdisplaybreaks[1]

\newtheorem*{namedtheorem}{\theoremname}
\newcommand{\theoremname}{testing}

\theoremstyle{plain}

\newtheorem{thm}{Theorem}[section]
\newtheorem{proposition}[thm]{Proposition}
\newtheorem{proposition-definition}[thm]{Proposition-Definition}
\newtheorem{lemma-definition}[thm]{Lemma-Definition}
\newtheorem{corollary}[thm]{Corollary}
\newtheorem{lemma}[thm]{Lemma}

\theoremstyle{definition}
\newtheorem{definition}[thm]{Definition}
\newtheorem{notation}[thm]{Notation}

\newtheorem{assumption}[thm]{Assumption}
\newtheorem{remark}[thm]{Remark}

\newtheorem{construction}[thm]{Construction}

\theoremstyle{remark}

\numberwithin{thm}{section}

\newcommand\ocM{\overline{\mathcal{M}}}

\newcommand\uGamma{\underline{\Gamma}}

\newcommand\cA{\mathcal{A}}

\newcommand\cC{\mathcal{C}}

\newcommand\cF{\mathcal{F}}
\newcommand\cG{\mathcal{G}}

\newcommand\cM{\mathcal{M}}

\newcommand\cO{\mathcal{O}}

\newcommand\cT{\mathcal{T}}

\def\O{\mathcal{O}}

\def\P{\mathbb{P}}
\def\A{\mathbb{A}}

\def\C{\mathbb{C}}

\newcommand\uC{\underline{C}}

\newcommand\uf{\underline{f}}

\newcommand\uS{\underline{S}}

\newcommand\uX{\underline{X}}
\newcommand\uY{\underline{Y}}
\newcommand\uZ{\underline{Z}}

\renewcommand\AA{\mathbb{A}}

\newcommand\CC{\mathbb{C}}

\newcommand\GG{\mathbb{G}}

\newcommand\NN{\mathbb{N}}
\newcommand\PP{\mathbb{P}}
\newcommand\QQ{\mathbb{Q}}

\newcommand\XX{\mathbb{X}}

\newcommand\ZZ{\mathbb{Z}}

\newcommand\fM{\mathfrak{M}}

\newcommand\fS{\mathfrak{S}}

\newcommand\fX{\mathfrak{X}}

\newcommand\tF{\tilde{F}}

\newcommand\arr{\ifinner\to\else\longrightarrow\fi}

\newcommand{\Gm}{\GG_m}

\def\displaytimes_#1{\mathrel{\mathop{\times}\limits_{#1}}}

\def\displayotimes_#1{\mathrel{\mathop{\bigotimes}\limits_{#1}}}

\newcommand\spec{\operatorname{Spec}}

\newcommand\sym{\operatorname{Sym}}

\newdir{ >}{{}*!/-5pt/@{>}}

\newcommand\doublelong[2]{\mathbin{\xymatrix{{}\ar@<3pt>[r]^{#1}
\ar@<-3pt>[r]_{#2}&}}}

\newlength{\ignora}

\renewcommand{\setminus}{\smallsetminus}

\newcommand{\Z}{\mathbb{Z}}

\numberwithin{equation}{subsection}

\newcommand{\X}{X} 

\newcommand{\Y}{Y} 

\newcommand{\ulY}{\underline{\Y}}

\newcommand{\ulX}{\underline{\X}}

\begin{document}

\title[Irreducibility and unirationality]{On the irreducibility of the space of genus zero stable log maps to wonderful compactifications}

\author{Qile Chen}

\author{Yi Zhu}

\address[Chen]{Department of Mathematics\\
Columbia University\\
Rm 628, MC 4421\\
2990 Broadway\\
New York, NY 10027\\
U.S.A.}
\email{q\_chen@math.columbia.edu}

\address[Zhu]{Department of Mathematics\\
University of Utah\\
Room 233\\
155 S 1400 E \\
Salt Lake City, UT 84112\\
U.S.A.}
\email{yzhu@math.utah.edu}

\thanks{Chen is partially supported by the Simons foundation and NSF grant DMS-1403271.}

\date{\today}


\begin{abstract}
In this paper, we prove the moduli spaces of genus zero stable log maps to {a large class of} wonderful compactifications of sober spherical homogeneous spaces are irreducible and unirational.
\end{abstract}
\maketitle

\tableofcontents

\section{Introduction}\label{sec:intro}

Throughout this paper, we work over an algebraically closed field of characteristic zero, denoted by $\CC$.

\subsection{Stable log maps}

Let $X$ be a fine and saturated log scheme. A {\em stable log map} to $X$ over a log scheme $S$ is given by the following data:
\[
(C \to S, f: C \to X, p_1, \cdots, p_n)
\]
where
\begin{enumerate}
 \item $C \to S$ is a family of log curve with markings $p_1, \cdots, p_n$, see \cite{FKato,LogCurve};
 \item $f: C\to X$ is a morphism of log schemes;
 \item the underlying map $\uf: \uC \to \uX$ obtained by removing log structures of $f$ is stable in the usual sense. 
\end{enumerate}
The {\em arrows} between stable log maps are defined by cartesian diagram in the category of log schemes. We thus obtain $\fM_{\Gamma}(X)$ the category of stable log maps with discrete data $\Gamma$ fibered over the category of log schemes. Here $\Gamma = (g, n, \beta, \{c_i\}_{i=1}^{n})$ consists of the genus $g$ of the source curve, the curve class $\beta$ of the underlying stable map, the number $n$ of marked points, and the contact orders $c_i$ for the $i$-th marked point. The contact order describes the tangency condition of the stable log map $f$ with the boundary of $X$ along the marked point. 

The theory of stable log maps has been established in a sequence of papers \cite{GS, Chen, AC, log-bound, wise}. The fibered category $\fM_{\Gamma}(X)$ is shown to be represented by a Deligne-Mumford stack with the natural fine and saturated log structure, and is proper when the underlying scheme $X$ is projective. 

\subsection{Main result}

The goal of the current paper is to prove the irreducibility and unirationality of the space of stable log maps of genus zero for wonderful compactifications \cite{Knop,Luna}.


\begin{thm}\label{thm:main}
Let $G$ be a semi-simple linear algebraic group over $\CC$, and $H \subset G$ be a sober spherical subgroup. Let $\X$ be {the} log smooth variety associated to the wonderful compactification of $G/H$. Assume that all colors of the wonderful compactification are of type (b). For any discrete data $\Gamma = (g=0,n,\beta,\{c_i\}_{i=1}^n)$, the coarse moduli space
 associated to the stack $\fM_{\Gamma}(\X)$ of stable log maps, {if it is non-empty,} is irreducible and unirational. 
\end{thm}


The irreducibility is proved in Proposition \ref{prop:irreducibility}, and the unirationality follows from Proposition \ref{prop:quotient-of-rational}. 


\begin{remark} 
In the above theorem, the colors are the $B$-stable divisors of $G/H$. According to Luna's study of wonderful compactifications \cite[Section 1.4]{Luna}, each color has a type: (a), (a') or (b), {see Section \ref{sec:b} and Definition \ref{def:b} for more details}. 

The class of wonderful compactifications with all colors of type (b) contains large number of interesting examples, including
\begin{itemize}
\item wonderful compactifications of semisimple groups;
\item spherical varieties of minimal rank \cite{Brion-evb,min-rank};
\item cuspidal spherical varieties of rank one with two exceptions: $(\P^2, \text{a smooth conic})$, and $(\P^1\times\P^1, \text{the diagonal }\P^1)$ \cite{Ahiezer,Brion-rankone}.
\end{itemize}
\end{remark}

\begin{remark} 
Our proof indeed shows that Theorem \ref{thm:main} holds for any curve class $\beta$ satisfying Assumption \ref{assu:good-curve}. When the type (b) condition holds, Assumption \ref{assu:good-curve} is automatic.

Assumption \ref{assu:good-curve} is to enforce the underlying maps of the general limiting stable log maps factor through the center $\uY$ of the target $\uX$, see Section \ref{sec:b} for the definition. In general, the uniqueness of Proposition \ref{prop:unique-log-lift} could fail without Assumption \ref{assu:good-curve}, and a different approach to the irreducibility may be  required. We hope to investigate this situation in our future work.
\end{remark}


\begin{remark}
The nonemptyness of the space of stable log maps of genus zero is completely solved in our paper \cite{A1}, where we give a classification of all discrete data $\Gamma$ for such stable log maps.
\end{remark}

\begin{remark}
Wonderful compactifications of rank zero are projective homogeneous spaces with the empty boundary and the ``type (b)" condition is automatic. In this case, the stack of stable log maps $\fM_{\Gamma}(\X)$ is the same as the stack of usual stable maps to $\underline{\X}$, and is proved to be irreducible and rational by Kim-Pandharipande \cite{KP} by considering the maximal degeneration via a general torus action.  The irreducibility is also proved independently by J. Thomsen in \cite{Thomsen}. The above result may be viewed as a generalization of their results in the logarithmic setting.
\end{remark}


\begin{remark}

In case $X$ is given by a toric variety with its toric boundary, the stack $\fM_{\Gamma}(X)$ of two-pointed genus zero stable log maps to $X$ is shown to have its coarse moduli given by the Chow quotient of $X$, see  \cite{CM}.
\end{remark}

\subsection{Organization of the paper} In this paper, we consider a general torus action on the stack $\fM_{\Gamma}(\X)$ induced via its action on the target $\X$. In Section \ref{sec:underlying-specialization}, we study the maximal degeneration of underlying stable maps in $\ulX$ under the give torus action. This is quite similar to the situation of \cite{KP}. 

However, not every underlying stable log map can be lifted to stable log maps, since the canonical morphism $\fM_{\Gamma}(\X) \to \fM_{\underline{\Gamma}}(\ulX)$ to the stack of underlying stable maps is not surjective in general. This is a major technique difficulty of stable log maps. In Section \ref{sec:log-limit}, we show that the maximal degenerate underlying stable map in question admits a unique logarithmic lifting. Assumption \ref{assu:good-curve} is crucial for the uniqueness of the lifting.

In Section \ref{sec:rationality}, we provide the proof of  Theorem \ref{thm:main} by investigating the maximal Bialynicki-Birula cell of $\fM_{\Gamma}(\X)$ under the given torus action. Note that the stack $\fM_{\Gamma}(\X)$ is only log smooth along the fixed locus, and could have toric singularities in general. We thus pass to an equivariant resolution of the Bialynicki-Birula cell. An analysis of the log structure along the fixed locus is given for the desingularization to be unirational.

A further study of genus zero log Gromov-Witten invariants for wonderful compactifications will require a much more detailed analysis of all Bialynicki-Birula cells. We wish to carry this out in our future work.

\subsection{Notations}\label{ss:notation} 

All log structures in this paper are assumed to be fine and saturated. We refer to \cite{KKato} for the basics of logarithmic geometry that will be used in this paper. Capital letters such as $C, S, X, Y$, and $Z$ are reserved for log schemes. The corresponding underlying schemes are denoted by $\uC, \uS, \uX, \uY$, and $\uZ$ respectively. 

Given a log scheme $X$, a stable log map to $X$ is usually denoted by $f: C/S \to X$ where $C \to S$ is the family of source log curves. When $\uS$ is a geometric point with the trivial log structure, we will simply write $f: C \to X$. {Recall the discrete data for stable log maps as we have introduced previously:}
\begin{equation}\label{equ:data-to-X}
\Gamma = (g=0, n, \beta, \{c_i\}_{i=1}^n),
\end{equation}
where $g$ denotes the genus of the source underlying curve, $n$ is a non-negative integer denotes the number of marked points, $\beta$ is a curve class in $\uX$, and $\{c_i\}_{i=1}^n$ is the collection of contact orders of the marked points along the boundary divisor of $X$. See Construction \ref{cons:monoid} for the contact orders needed in the paper. We refer to \cite[Section 4.1]{AC} and \cite{ACGS} for more details of contact orders.

Let $\fM_{\Gamma}(X)$ be the stack of $n$-marked, genus zero stable log maps with curve class $\beta \in H_2(X)$, and contact order $c_i$ along the $i$-th marking. Write $\uGamma = (g=0, n, \beta)$ by removing contact orders from $\Gamma$.

\bigskip

\subsection*{Acknowledgments}
The authors are grateful to Dan Abramovich, Jason Starr, Michael Thaddeus, and Xinwen Zhu for their useful discussions. A large part of our work has been done while the first author's visit of the Math Department in the University of Utah in February 2014. We would like to thank the Utah Math Department for its hospitality.

\section{Specialization of underlying stable maps}\label{sec:underlying-specialization}

\subsection{Torus invariant curves and divisors on $\ulX$}\label{ss:BB}

Throughout this paper, we fix a simply connected semisimple linear algebraic group $G$ over $\CC$. Denote by $\fX^* := \fX^*(T)$  and $\fX_{*} = \fX_{*}(T)$ the character and cocharacter groups of $T$ respectively.

Consider a sober spherical subgroup $H \subset G$. This means that $G/H$ is a spherical variety and $N_G(H)/H$ is finite. Let $\ulX$ be the wonderful compactification of $G/H$. Denote by $\X$ the log smooth varieties associated to the pair $(\ulX,  \ulX\setminus G/H)$ with the natural $G$-action, see \cite[Proposition 5.4]{A1}.

We first recall the basic setup regarding the Bialynicki-Birula stratification of $\uX$ as in \cite[Section 1]{Brion-onecycle}. For the rest of this paper, we fix a regular one-dimensional torus $\lambda: \Gm \to T$ such that the set of fixed points $\ulX^{\lambda}$ equals $\ulX^{T}$. Choose a Borel subgroup $B$ containing $T$ such that $$B=G(\lambda):=\{ g\in G|\lim_{t\to 0}\lambda(t)g\lambda(t^{-1})\in G\}.$$ Let $B^-$ be the opposite Borel subgroup of $B$. 

Let $\ulY$ be the unique closed $G$-orbits of $\ulX$. Hence the set $\uX^{B^-}$ consists of a unique point lying on $\uY$, denoted by $x^-$. Denote by $\Y \subset \X$ the strict closed sub-log schemes with the underlying structure $\uY$.  

For any closed subscheme $\uZ \subset \ulX^{\lambda}$, we introduce two subsets of $\uX$:
\[
\ulX^{+}(\uZ) = \{ x \in \ulX \ | \ \lim_{t \to 0} \lambda(t)\cdot x  \in \uZ\}
\]
and 
\[
\ulX^{-}(\uZ) = \{ x \in \ulX \ | \ \lim_{t \to \infty} \lambda(t)\cdot x  \in \uZ\}.
\]

\begin{lemma}[\cite{Brion-onecycle}, Lemma 2]\label{lem:open}
We have the following:
\begin{enumerate}
\item $\ulX^{-}(x^{-}) \subset \ulX$ is open. We write $\ulX^{-} := \ulX^{-}(x^{-})$ in short. 
\item $\uX^-$ is an open affine $B$-invariant neighborhood of $x^-$.
\item $\GG_m$ acts on the algebra of regular functions $\O_{\uX^-}$ with positive weights. Thus the $\GG_m$-action extends to a morphism:
$$\A^1\times \uX^-\to \uX^-.$$
\end{enumerate}
\end{lemma}



\begin{lemma} \label{lem:color}
The colors of $\uX$ are the irreducible components $D_{1}, \cdots, D_l$ of the closed subset $\ulX \setminus \ulX^-$. In particular, they are $B$-stable but not $G$-stable.
\end{lemma}

\proof Let $\uX_{\uY,B}$ be the $B$-invariant open subset of $\uX$ by deleting all colors. By \cite[Proposition 2.2]{Brion-spherical}, we have $$\uX_{\uY,B}=\{x\in\uX\  | \ \overline{B x}\supset Y\}.$$
Thus it suffices to show that $\uX_{\uY,B}=\uX^-$. The inclusion $\uX^-\subset\uX_{\uY,B}$ is clear because the orbit $Bx^-$ is dense in $\uY$. The other inclusion follows from that the complement of $\uX^-$ is a union of colors, c.f., \cite[Theorem 1]{Brion-onecycle}.
\qed

\begin{thm}[\cite{Brion-onecycle}, Theorem 1]
The set of colors $D_1,\cdots, D_l$ are globally generated cartier divisors, which form a $\Z$-basis of the Picard group of $\ulX$. Furthermore, they generate the nef cone of $\uX$.\qed
\end{thm}

For wonderful compactifications, we strengthen a result of Brion \cite[Theorem 2]{Brion-onecycle}.

\begin{thm}\label{prop:curve-class} When $\uX$ is a wonderful compactification, we have:
\begin{enumerate}
\item For each $i=1,\cdots, l$, there exists a unique point $x_i^{-}\in D_i^{\lambda}$ such that $D_i$ is the closure of $X^{-}(x_i^-)$.
\item The $B^-$-invariant curve $P_i=\overline{B^-x_i^-}$ intersects $D_j$ at the unique point $x_i^-$, and intersects no other $D_i$. All $P_i$ are isomorphic to $\P^1$.
\item $P_i$ intersects $D_i$ transversally at $x_i^-$. Thus $(P_i.D_j)=\delta_{ij}$ and the classes $P_1,\cdots,P_l$ form a $\Z$-basis of $N_1(\uX)$.
\end{enumerate}

\end{thm}

\proof Brion proves (1), (2) in \cite[Theorem 2]{Brion-onecycle}. He also proves (3) when $\uX$ is nonsingular. It remains to prove the case when $\uX$ is singular. Here the trick is to use Luna's idea on spherical closure \cite[Section 6.1]{Luna}. Let $H'\supset H$ be the spherical closure of $H$ and let $\uX'$ be the wonderful compactification of $G/H'$. By \cite[Remark 30.1]{Timashev}, there exists a $G$-equivariant morphism 
$$\pi:\uX\to\uX'$$
such that $\uX'$ is smooth \cite[Corollary 7.6]{Knop96} and the set of colors on $\uX$ is identified with those of $\uX'$ via pullback along $\pi$. 
Now (3) follows from its nonsingular case and the projection formula.\qed



\subsection{Type (b) condition}\label{sec:b}
In general, the torus fixed points $x_{i}^{-}$ for $j = 1, \cdots, l$ do not necessarily lie in the closed orbit $\uY$. One may for example, consider the wonderful compactification corresponding to the pair $(\P^2, \text{a smooth conic})$. Next we give a sufficient and necessary criterion to rule out such situation.


We briefly recall the type of a color on wonderful compactifications \cite{Luna}. Let $\Sigma$ be the finite set of spherical roots of $\uX$, lying in the character lattice $\XX^*(T)$. See \cite[Proposition 6.4]{Luna}.

We say that a simple root $\alpha$ \emph{moves} a color $D$ if $P_\alpha D\neq D$ where $P_\alpha$ is the minimal parabolic group containing $B$ and associated to $\alpha$. Each color $D$ is moved by a unique simple root $\alpha_D$. Let $\Delta(\alpha)$ be the set of colors moved by $\alpha$.  
\begin{definition}\label{def:b}
We say that the color $D$ is 
\begin{enumerate}
\item of type (a) if $\Delta(\alpha_D)$ contains two colors;
\item of type (a') if $\Delta(\alpha_D)=\{D\}$ and $2\alpha\in \Sigma$;
\item of type (b) if $\Delta(\alpha_D)=\{D\}$ and no multiple of $\alpha$ is in $\Sigma$.
\end{enumerate}

\end{definition}
By \cite[Section 1.4]{Luna}, each color belongs to a unique type as above.

\begin{proposition}\label{prop:color}
A color $D_i$ is of type (b) if and only if the fixed point $x_i^-$ lies in the closed orbit $\uY$.
\end{proposition}
\proof Let $\alpha$ be the simple root moving $D_i$. Consider the localization $\uX^\alpha$ of $\uX$ at $\alpha$, that is, $\uX^\alpha$ is the $T_\alpha$-fixed points of $P_\alpha\uX_{Y,B}$, where $P_\alpha$ is the standard minimal parabolic group determined by $\alpha$, and $T_\alpha$ is the connected component of the center of the Levi subgroup $L(P_\alpha)$. By \cite[Lemma 30.2]{Timashev}, $\uX^\alpha$ is a wonderful compactification of rank less than one with the action of either $SL_2$ or $PGL_2$, and $x_i^-$ is the $\lambda$-torus fixed point where a generic point of the color $D_i\cap \uX^\alpha$ retracts to. By the classification result in \cite[p. 216]{Timashev}, if $\alpha$ is of type (b), $\uX^\alpha$ is a projective homogeneous space isomorphic to $\P^1$, which implies that $x_i^-$ lies in the closed orbit $\uY$. If $\alpha$ is not of type (b), $\uX^\alpha$ is isomorphic to the pair $(\P^2,\text{conic})$ or $(\P^1\times\P^1, \text{the diagonal})$. In these two cases, $x_i^-$ does not lie in the unique closed orbit.
\qed

\subsection{A specialization of non-degenerate underlying stable maps}\label{ss:underlying-deg}

Let $U$ be the union
\begin{equation}\label{equ:open-stratum}
\ulX^-\cup X^-({x_1^-})\cup\cdots\cup X^-({x_l^-}) .
\end{equation} with the notations as in Theorem \ref{prop:curve-class}.
By the Bialynicki-Birula stratification, $U$ is an open subscheme of $\ulX$ whose completement is of codimention at least $2$.

We consider the following situation:

\begin{notation}\label{not:good-map}
Let $\uf: \uC \to \ulX$ be a genus zero usual stable map with discrete data $\uGamma$ induced by $\Gamma$ as in (\ref{equ:data-to-X}) with the following properties:
\begin{enumerate}
\item $\uC \cong \PP^1$, and $\uf(\uC) \cap G/H \neq \emptyset$.
\item $\uf(\uC) \subset U$, with $U$ given by (\ref{equ:open-stratum}). 
\item $\uf(\uC)$ intersects the divisors $D_i$ transversally at distinct non-marked points. 
\item The images of all markings under $\uf$ lie in $\ulX^-$.
\end{enumerate} 
In particular, the stable map $\uf$ lifts to a unique stable log map $f: C \to \X$.
\end{notation}

Consider the usual stable map $\uf: \uC \to \ulX$ as in Notation \ref{not:good-map}.  We may assume that
\begin{equation}\label{equ:curve-class-ad}
\beta = \sum_{i=1}^{l}k_{i}[P_{i}],
\end{equation}
with non-negative integers $k_{i}$ for all $i$. 

Consider the general $\Gm$-action $\lambda$ as in the beginning of Section \ref{ss:BB}. Denote by $\uf_0: \uC_0 \to \ulX$ the limit of the underlying stable map $\uf$ when $t \to \infty$ under the $\Gm$-action. 

\begin{notation}\label{not:limit}
To construct such $\uf_0$, we introduce a underlying pre-stable map $\tF: \uC_{1} \to \uX$ as follows:
\begin{enumerate}
 \item $\uC_{1} = \uC \cup \bigcup_{i=1}^{l}(\cup_{j=1}^{k_i}\PP_{i,j})$, where $\PP_{i,j} \cong \PP^1$ is attached to $\uC$ at $x_{i,j}$;
 \item All marked points on $\uC_1$ is given by the markings on $\uC$;
 \item $\tF(\uC) = x^-$, and $\tF|_{\PP_{i,j}}: \PP_{i,j} \to P_{i}$ with $\tF_{*}[\PP_{i,j}] = [P_{i}]$.
\end{enumerate} 
\end{notation}

\begin{lemma}\label{lem:adjoint-type-limit}
Let $\uf$ be the stable map as in Notation \ref{not:good-map}. Then the limit underlying stable map $\uf_0$ is given by the stabilization of $\tF$.
\end{lemma}
\begin{proof}
The proof is similar to that of  \cite[Proposition 2]{KP}. For completeness, we record it below. Consider the family of underlying stable maps
\[
h: \Gm \times \uC \to \X
\]
given by the $\Gm$-action. By Lemma \ref{lem:open}, we could extend $h$ to a morphism of usual schemes:
\[
h: \AA^1 \times (\uC \setminus \{x_{i,j}\}) \to \X.
\]
Let $S \to \AA^1\times \uC$ be a suitable blow-up along the isolated nonsingular points $\{0\times x_{ij}\}$, we obtain a morphism
\[
h': S \to \X.
\]
 
Observe that the limit of $t\cdot f(x_{i,j})$ is $x_i^{-}$ as $t \mapsto \infty$. Hence the exceptional divisor $E_{i,j}$ of $S \to \AA^1\times \uC$ over $x_{i,j}$ connects the two points $x_{i}^{-}$ and $x^{-}$ under the map $h'$. 

Note that the torus action does not change the cross-ratio of the set of points $\{x_{i,j}\}$. By degree consideration, the curve class of $h'|_{E_{i,j}}$ is $[P_i]$, and $h'(E_{i,j}) = P_i$ as $P_i$ is the unique fixed curve joining $x_i^{-}$ and $x^-$ with respect to the chosen $\GG_m$-action. 

After possibly further blow-ups and base changes over the fiber of $0 \in \AA^1$, we may assume that $S$ is non-singular, and each $E_{i,j}$ is a normal crossings divisor. Thus, each $E_{i,j}$ has a single component mapped to $P_i$ isomorphically, and all other components get contracted by $h'$.

Blowing-down the $h'$-contracted components of each $E_{i,j}$, we obtain a morphism $h'': S' \to \X$ which is a family of nodal curves over $\AA^1$. From the above construction, the fiber $h''|_{0}$ over $0 \in \AA^1$ is the stabilization of $\tF$ as required.


\end{proof}

We then summarize our discussion as follows:

\begin{proposition}\label{prop:limit}
Let $\uf$ be the underlying stable map as in Notation \ref{not:good-map}. Then the limiting underlying stable map $\uf_0$ is given by one of the following situations:
\begin{enumerate}
 \item $\tF$ is stable, and $\uf_0 = \tF$; or
 \item $\sum_i k_i = 1$, $n \leq 1$, and $\uf_0$ is given by $\PP^1 \to P_i$ with the marking mapped to $x^-$; or
 \item $\sum_i k_i = 2$, $n = 0$, and $\uf_0: \PP^1\cup\PP^1 \to P_i\cup P_j$ for some $i,j$, with the unique node mapped to $x^-$.
\end{enumerate}
\end{proposition}
\begin{proof}
This follows from Lemma \ref{lem:adjoint-type-limit}, and taking the stabilization of $\tF$.
\end{proof}



Finally we introduce the following condition for the curve class $\beta$ as in (\ref{equ:curve-class-ad}):

\begin{assumption}\label{assu:good-curve}
$k_i = 0$ as long as the color $D_i$ is not of type (b).
\end{assumption}

\begin{lemma}\label{lem:curve-in-center}
Under Assumption \ref{assu:good-curve}, the underlying stable map $\uf_0$ factors through the center $\uY$.
\end{lemma}
\begin{proof}
This follows from the fact that the curve $P_{i} = \overline{B\cdot x_{i}^{-}}$ lies in $\uY$ as in Theorem \ref{prop:curve-class} and Proposition \ref{prop:color}.
\end{proof}

\section{Specialization of stable log maps}\label{sec:log-limit}

\subsection{Uniqueness of the lifting}

We have analyzed the specialization of the underlying stable map $\uf$ in Notation \ref{not:good-map} under the given $\GG_m$-action. Note that the underlying stable map $\uf$ uniquely determines a stable log map $f$. Thus by the properness of $\fM_{\Gamma}(\X)$, such limit as stable log maps exists, and is unique. But for our purposes, we need to understand all possible liftings over the limiting underlying stable map in Notation \ref{not:limit} as a stable log maps.  We first show that

\begin{proposition}\label{prop:unique-log-lift}
Let $\uf_0$ be a underlying stable map given by the stabilization of $\tF$ in Notation \ref{not:limit}. Under Assumption \ref{assu:good-curve}, there exists up to a unique isomorphism at most one minimal stable log map $f_0$, whose underlying stable map is given by $\uf_0$. 
\end{proposition}

Assume such lifting $f_0$ is given. We first calculate the possible characteristic monoid $\ocM_{S}$ with the given underlying stable map $\uf_0$. Denote by $\ocM := \ocM_{\Y,y}$ for any point $y \in \Y \subset \X$. Recall from \cite[Proposition 5.8]{A1} that there is a global morphism from the globally constant sheaf of monoids:
\begin{equation}\label{equ:chart}
\gamma: \ocM \to \ocM_{\X}
\end{equation}
such that
\begin{enumerate}
 \item $\gamma$ lifts to a chart of $\cM_{\X}$ Zariski locally on $\X$;
 \item the restriction $\gamma|_{\Y}: \ocM \to \ocM_{\Y}$ is an isomorphism of sheaves of monoids.
\end{enumerate}

The minimal monoid is defined in the Deligne-Faltings case in \cite[Section 4.1]{AC} and for Zariski log structures in \cite[Construction 1.16]{GS}. In what follows, we will mimic \cite[Section 4.1]{AC}, and reformulate the minimal monoid in a slightly different manner for the convenience of our argument. It should be straight forward to verify that the following is equivalent to the definitions in \cite{AC,GS} in our particular case.

\begin{construction}\label{cons:monoid}
Denote by $\underline{\Phi}$ the dual graph of the underlying curve $\uC_0$. This means that $\underline{\Phi}$ is a connected graph with the set of vertices $V(\underline{\Phi})$ given by the irreducible components of $\uC_0$ and the set of edges $E(\underline{\Phi})$ given by the nodes of $\uC_0$. For each $v \in V(\underline{\Phi})$ and $e \in E(\underline{\Phi})$, denote by $\uZ_v$ and $p_v$ the irreducible component and node respectively. Denote by $\Phi$ the {\em marked graph} obtained by decorating $\underline{\Phi}$ with the following data:
\begin{enumerate}
 \item For each $v \in V(\underline{\Phi})$ we associate the monoid $\ocM_v := \ocM$.

 \item For each $e \in E(\underline{\Phi})$ we associate the free monoid $\NN_e \cong \NN$. We use $e$ to denote the generator of $\NN_e$.

 \item We fix an orientation on the graph as follows. If $\uf_0 = \tF$, then each edge $e$ is oriented from the teeth $\PP_{ij}$ to the handle $\uC$. Otherwise, $\uC_0$ has at most two components, in which case we fix an arbitrary orientation of the unique edge.

 \item For each edge $e$ orienting from $v_1$ to $v_2$, denote by $c_e \in (\ocM^{\vee})^{gp}$ the $\gamma$-contact order of the one cycle $(\uf_0)_*Z_{v_1}$ as in \cite[Definition 3.4]{A1}.  In particular, $c_e$ defines a group morphism $\ocM^{gp} \to \ZZ$.
\end{enumerate}

Now for each edge $e$ orienting from $v_1$ to $v_2$, and each element $\delta \in \ocM^{gp}$, we introduce the relation
\begin{equation}\label{equ:edge-relation}
\delta_{v_1} + c_e(\delta)\cdot e = \delta_{v_2}
\end{equation}
where $\delta_{v_i}$ denotes the element in $\ocM_{v_i}^{gp}$ given by $\delta$. Denote by $\ocM(\Phi)^{gp}$ the lattice given by $\sum_{e}\NN^{gp}_e \oplus \sum_{v}\ocM^{gp}_v$ modulo the relations (\ref{equ:edge-relation}) for all $e$. We thus have a natural morphism 
\begin{equation}\label{equ:canonical-mon-elm}
\psi: \sum_{e}\NN_e \oplus \sum_{v}\ocM_v \to \ocM(\Phi).
\end{equation}
Denote by $\ocM(\Phi)_{\QQ}^{+}$ the rational cone generated by the image of $\phi$ in $\ocM(\Phi)_{\QQ} := \ocM(\Phi)\otimes_{\ZZ}\QQ$. Then we write
\begin{equation}\label{min:monoid}
\ocM(\Phi) := \ocM(\Phi)^{gp} \cap \ocM(\Phi)_{\QQ}^{+}.
\end{equation}
\end{construction}

\begin{lemma}\label{lem:min-monoid}
Notations and assumptions as above, assume $\uf_0$ lifts to a minimal stable log map $f_0$. Then we have
\begin{enumerate}
 \item $\ocM_{S} = \ocM(\Phi)$, and
 \item a natural splitting $\ocM(\Phi)^{gp} = \ocM^{gp} \oplus \sum_{e}\NN_e^{gp}$.
\end{enumerate}
\end{lemma}
\begin{proof}
Notice that the possibilities of $\uf_0$  are listed in Proposition \ref{prop:limit}. For an edge $e$ orienting from $v_1$ to $v_2$, there is no marking on $\uZ_{v_1}$. Thus, the contact order of the node $p_e$ is given by $c_e$. Now the first statement follows from the definition of minimality in \cite[Section 4]{AC} or \cite[Construction 1.16]{GS}.

Consider the second statement. We first assume that $\uf_0 = \tF$. Consider the vertex $v_0$ associated to the handle of $\uC_0$. Then for any other vertex $v$, there is a unique edge $e$ orienting from $v$ to $v_0$. Then the relation (\ref{equ:edge-relation}) uniquely expresses $\delta_{v} = \delta_{v_0} - c_e(\delta)\cdot e$ for any $\delta$, which proves (2) in this case. The two other cases in Proposition \ref{prop:limit} can be proved similarly.
\end{proof}

\begin{proof}[Proof of Proposition \ref{prop:unique-log-lift}]
By Lemma \ref{lem:min-monoid} and (\ref{equ:canonical-mon-elm}), we have the canonical map of monoids:
\[
\psi_e: \sum_e \NN_e \to \ocM(\Phi).
\]
Denote by $C_0^{\sharp} \to S^{\sharp}$ the canonical log curve associated to the underlying curve $\uC_0$. We then fix a chart $\gamma^{\sharp}: \sum_e \NN_e \to \cM_{S^{\sharp}}$. Since $C_0 \to S$ is obtained by pulling back $C_0^{\sharp} \to S^{\sharp}$, and $\uS$ is a geometric point, we may assume that 
\begin{equation}\label{equ:log-curve}
\cM_S = \ocM(\Phi)\oplus_{\psi_e,\sum_e \NN_e, \gamma^{\sharp}}\cM_{S^{\sharp}}.
\end{equation}
This naturally associated with a morphism of log structures $\cM_{S^{\sharp}} \to \cM_S$, which is unique up to a unique isomorphism by Lemma \ref{lem:min-monoid}(2). This defines the log curve $C_0 \to S$ up to a unique isomorphism together with a chart $\gamma_S: \ocM(\Phi) \to \cM_S$. 

Now assume that we have two liftings $f_{01}: C_0/S \to \X$ and $f_{02}: C_0/S \to \X$ over $\uf_0$. We need to verify that $f_{01}$ and $f_{02}$ is differ by an isomorphism of $C_0/S$. We first notice that the two morphisms $\bar{f}_{01}^{\flat} = \bar{f}_{02}^{\flat}: \uf_0^{*}\ocM_{\X} \to \ocM_{C_0}$ on the level of characteristic monoids coincide. This is because the discrete data of $f_{01}$ and $f_{02}$ are given by  the same marked graph $\Phi$. We may denote both $\bar{f}_{01}^{\flat}$ and  $\bar{f}_{02}^{\flat}$ by $\bar{f}_0^{\flat}$. 

Next consider the quotients $q_{\X}: \cM_{\X} \to \ocM_{\X}$ and $q_{C_0}: \cM_{\C_0} \to \ocM_{X_0}$. 
For any $\delta \in \ocM$, denote by 
\begin{equation}\label{equ:torsors}
\cT_{\X}(\delta) := q_{\X}^{-1}(\gamma(\delta)) \ \ \ \mbox{and} \ \ \ \cT_{C_0}(\delta) := q_{C_0}^{-1}(\bar{f}^{\flat}_0\circ \gamma(\delta))
\end{equation} 
the two $\cO^*$-torsors over $\uC_0$. For each $i$, we obtain an isomorphism of torsors:
\begin{equation}\label{equ:torsor-iso}
f_{0i}^{\flat}|_{\cT_{\X}(\delta)}: \cT_{\X}(\delta) \to \cT_{C_0}(\delta),
\end{equation}
induced by the corresponding morphisms of log structures. Since those are isomorphisms of the same torsors over a proper curve, we have
\begin{equation}\label{equ:torsor-equal}
f_{01}^{\flat}|_{\cT_{\X}(\delta)} = \log u_{\delta} + f_{02}^{\flat}|_{\cT_{\X}(\delta)}
\end{equation}
for a non-zero constant $u_{\delta}$ uniquely determined by $\delta$, $f_{01}$, and $f_{02}$. This defines a map of sets:
\[
h^{gp}: \ocM^{gp} \to \CC^*.
\]
In particular, $h$ is a morphism of groups. 

By Lemma \ref{lem:min-monoid}(2), the chart $\gamma_S$ induces a morphism of groups
\[
\gamma_S^{gp}: \ocM^{gp} \oplus \sum_{e}\NN_e^{gp} \to \cM_{S}
\]
Then we obtain a morphism of groups
\[
\tilde{h}^{gp}: \cM_{S}^{gp} \to \cM_{S}^{gp}
\]
such that $\tilde{h}^{gp}|_{\cO^*} = id_{\cO^*}$, $\tilde{h}^{gp}|_{\gamma_S^{gp}(\sum\NN_e)} = id_{\gamma_S^{gp}(\sum\NN_e)}$, and $\tilde{h}^{gp}(\gamma_{S}^{gp}(\delta)) = \log h(\delta) + \gamma_{S}^{gp}(\delta)$ for any $\delta \in \ocM$. Thus the restriction $\tilde{h} := \tilde{h}^{gp}|_{\cM_{S}}$ defines an isomorphism of the log structure $\cM_{S}$. Furthermore, since $\tilde{h}$ fixes the image $\gamma_{S}(\sum_e\NN_e)$, it induces an isomorphism of the log curve $C_0/S$. Denote by $\tilde{h}_{C_0}: \cM_{C_0} \to \cM_{C_0}$ the isomorphism induced by $\tilde{h}$. Then (\ref{equ:torsor-equal}) and the construction of $\tilde{h}$ imply that 
\[
f_{01}^{\flat} = f_{02}^{\flat} \circ \tilde{h}_{C_0}.
\]
Thus the log maps $f_{01}$ and $f_{02}$ is differ by an isomorphism of $C_0/S$. The uniqueness of $\tilde{h}$ follows from the uniqueness of $u_{\delta}$ as in (\ref{equ:torsor-equal}).
\end{proof}

\subsection{Existence of the lifting}

The proof of Proposition \ref{prop:unique-log-lift} can be modified to show the existence of lifting as follows:

\begin{lemma}\label{lem:exist-log-lift}
Assume $\fM_{\Gamma}(\X) \neq \emptyset$, and the curve class $\beta$ satisfies Assumption \ref{assu:good-curve}. For each underlying stable map $\uf_0$ given by the stabilization of $\tF$ in Notation \ref{not:limit} with discrete data $\underline{\Gamma}$, there exist a unique stable log map $[f_0] \in \fM_{\Gamma}(\X)$ above $\uf_0$ with discrete data $\Gamma$. 
\end{lemma}
\begin{proof}
The proof of this statement is similar to the construction in \cite[Section 4.2]{A1}.

We first observe that there exists a stable log map realizing the marked graph $\Phi$. Indeed, this follows from the condition $\fM_{\Gamma}(\X) \neq \emptyset$, the properness of $\fM_{\Gamma}(\X) \neq \emptyset$, and the specialization of non-degenerate stable maps in Section \ref{ss:underlying-deg}.

Notice that the discrete data of $f_0$ is uniquely determined by the marked graph $\Phi$ associated to underlying stable map $\uf_0$ as in Construction \ref{cons:monoid}. This implies that the monoid in (\ref{min:monoid}) is sharp.  The same construction in (\ref{equ:log-curve}) uniquely determines the log curve $C_0/S$ up to a unique isomorphism. Since the underlying stable map is given, it remains to construct morphism of log structures $f_0^{\flat}: \uf_0^*\cM_{\X} \to \cM_{C_0}$. Since on the level of characteristic sheaves monoids, the morphism $\bar{f}_0^{\flat}: \uf_0^*\ocM_{\X} \to \ocM_{C_0}$ has been determined by the graph $\Phi$, it sufficies to construct $(f_0^{\flat})^{gp}: \uf_0^*\cM_{\X}^{gp} \to \cM_{C_0}^{gp}$.

Note that the discrete data determines $\bar{f}_0^{\flat}$ hence the torsors as in (\ref{equ:torsors}). Similarly as in (\ref{equ:torsor-iso}), given $(f_0^{\flat})^{gp}$ is equivalent to construct isomorphisms of torsors for each $\delta \in \ocM^{gp}$. Hence to construct $(f_0^{\flat})^{gp}$, it suffices to find a collection of isomorphisms of torsors for each $\delta $ in a basis of $\ocM^{gp}$. Note that an isomorphism of two $\cO^*$-torsors over genus zero curves  exists if and only if the degree of the torsors are compatible component-wise. Such compatibility only depends again on the discrete data in Construction \ref{cons:monoid}, and follows from the existence of stable log maps realizing the marked graph $\Phi$. This proves the existence. 

\end{proof}

\section{Irreducibility and unirationality}\label{sec:rationality}

\subsection{The irreducibility}

\begin{lemma}\label{lem:group-move}
Let $f: C \to \X$ be a genus zero stable log map with discrete data $\Gamma$. Assume that $\uC$ is irreducible, and $f(C) \cap G/H \neq \emptyset$. Then there is a nonempty open subset $V \subset G$ such that for any $g \in V$, the composition $g \circ f$ is a $\Gamma$-stable log map with the underlying map $\underline{g\circ f}$ satisfying the conditions in Notation \ref{not:good-map}.
\end{lemma}
\begin{proof}
First consider the open curve $C^{\circ} = C \setminus \{q_j\}$ by removing markings with non-trivial contact orders. Then the restriction $f|_{C^{\circ}}$ factors through $G/H \subset \X$. By Kleiman-Bertini Theorem \cite[Theorem 10.8]{Hartshorne}, there is an open dense $V \subset G$, such that for any $g \in V$ the restriction $g\circ f|_{C^{\circ}}$ satisfies the conditions (2) and (3) in Notation \ref{not:good-map}.

We notice that for each contact marking $q_j$, its image lies in a $G$-orbit, say $O_j$. Since the complement of $\ulX^{-}\cap O_j$ in $O_j$ is of codimension greater or equal than one. Applying Kleiman-Bertini Theorem again,  condition (4) can be achieved for contact markings by further shrinking $V$.
\end{proof}

Consider the following condition
\begin{equation}\label{equ:enough-point}
n + \sum k_i \geq 4.
\end{equation}
Under this assumption, the map $\tF$ in Notation \ref{not:limit} is stable. Consider the quasi-projective variety 
\begin{equation}\label{equ:rigidified-handle-moduli}
M = M_{0, n + \sum_{i}k_i}
\end{equation} 
parameterizing $( n + \sum_{i}k_i)$-distinct points over a smooth genus zero curves. The markings are labeled by
\[
p_{1}, \cdots, p_{n}; x_{1,1}, \cdots, x_{1,k_1}; \cdots ; x_{l,1}, \cdots, x_{l,k_l}.
\]

Let $\fS_{k_i}$ be the symmetric group acting on $M$ by permuting the markings $x_{i,1}, \cdots, x_{i, k_i}$. Denote by $B = M / \fS$ the quotient with $\fS$ given by the product
\begin{equation}\label{equ:permuting-marking}
\fS = \fS_{k_1}\times \cdots \times \fS_{k_l}.
\end{equation}

As observed in \cite[Section 4]{KP}, when $\sum_i k_i \geq 3$ the quotient $B$ is birational to the quotient
\[
\PP(\sym^{k_1}V^*)\times \cdots \times \PP(\sym^{k_l}V^*) // \mathbf{PGL}(V),
\]
where the latter is rational by \cite{Bog, Ka}.

\begin{proposition}\label{prop:attractor}
Assume $\fM_{\Gamma}(\X) \neq \emptyset$ and (\ref{equ:enough-point}). Under Assumption \ref{assu:good-curve}, there is a unique locally closed embedding $\cF: B \to \fM_{\Gamma}(\X)$ which sends a marked genus zero curve $\uC$ to the stable log map $f_0$ over the underlying stable map $\tF$ as in Notation \ref{not:limit} with handle $\uC$.
\end{proposition}
\begin{proof}
Let $\fM(\Phi) \subset \fM_{\Gamma}(\X)$ be locally closed substack consisting of stable log maps with the marked graph $\Phi$ as in Construction \ref{cons:monoid}. Consider the stratification of $\fM_{\Gamma}(\X)$ with respect to its log structure as in \cite[Lemma 3.5 (ii)]{LogStack}. Then we observe that $\fM(\Phi)$ is a stratum of $\fM_{\Gamma}(\X)$ such that the characteristic sheaf $\ocM_{\fM_{\Gamma}(\X)}$ is locally constant along $\fM(\Phi)$. Since $\fM_{\Gamma}(\X)$ is log smooth, the stratum $\fM(\Phi)$ has smooth underlying structure. 

Let $\fM_{\lambda,\infty} \subset \fM(\Phi)$ be the torus fixed locus consisting of the limits of general stable log maps as in Notation \ref{not:good-map} with respect to the chosen the torus action. Since $\fM(\Phi)$ has smooth underlying structure, $\fM_{\lambda,\infty}$ also has smooth underlying structure. Lemma \ref{lem:exist-log-lift} implies the tautological morphism 
\[\cG: \fM_{\lambda,\infty} \to B\] 
between two smooth stacks are one-to-one on the level of closed points.

We then observe that $\fM_{\lambda,\infty}$ is irreducible. Indeed, since $\fM_{\Gamma}(\X)$ is of finite type, and $\fM_{\lambda,\infty}$ is smooth, there is a unique irreducible component $\fM'_{\lambda,\infty} \subset \fM_{\lambda,\infty}$ such that $\cG|_{\fM'_{\lambda,\infty}}$ is dominant. Assume we have an object $[f] \in \fM_{\lambda,\infty}\setminus \fM_{\lambda,\infty}'$ given by a closed point. Since $B$ is irreducible, and $\cG$ is one-to-one on the level of closed points, there is an object $f_{S}$ over $S = \spec R$ where $R$ is a DVR, $\eta \in S$ is the generic point, and $s \in S$ is the closed point, such that $f_{\eta}$ is an object in $\fM_{\lambda,\infty}'$ and the central fiber $f_s$ has the underlying map given by that of $f$. It follows from Lemma \ref{lem:exist-log-lift} that $f_s = f$. This implies that $\fM'_{\lambda,\infty} = \fM_{\lambda,\infty}$.

Denote by $M \subset \fM_{\lambda,\infty}$ the fiber over the automorphism free locus of $B$. Then the restriction $\cG|_{M}$ is representable by the representability of \cite[Corollary 3.13]{AC}, since in this case the underlying stable maps of the log maps in $\fM_{\lambda,\infty}$ is automorphism free.

The cases when $B$ has stacky locus only occur when the equality in (\ref{equ:enough-point}) holds. Since $B$ parametrizes the underlying stable maps of the the log maps in $\fM_{\lambda,\infty}$, the stackyness of $B$ corresponding to the automorphism of the underlying stable maps. Note that this automorphism comes from subgroups of (\ref{equ:permuting-marking}), hence perserves the discrete data of $\Phi$. Thus we could lift the automorphism of the underlying stable maps to the corresponding log stable maps.  One could also see the existence of such lifting of automorphisms from the quotient construction using the rigidification (\ref{equ:rigidify-marking}) in the next section. Thus the morphism $\cG$ induces an isomorphism of the automorphism groups. This implies that $\cG$ is an isomorphism, whose inverse is the embedding $\cF$ as in the statement.
\end{proof}

\begin{proposition}\label{prop:irreducibility}
Under Assumption \ref{assu:good-curve} for the curve class $\beta$, the moduli space $\fM_{\Gamma}(\X)$ is irreducible.
\end{proposition}
\begin{proof}
Since $\fM_{\Gamma}(\X)$ is log smooth, the irreducibility of $\fM_{\Gamma}(\X)$ is equivalent to the connectedness. 

Notice that the open substack $\fM^{\circ}_{\Gamma}(\X) \subset \fM_{\Gamma}(\X)$ consisting of the points with trivial log structure is dense in $\fM_{\Gamma}(\X)$. Thus, by Lemma \ref{lem:group-move}, any stable log map $[f'] \in \fM_{\Gamma}(\X)$ deforms to a stable log map $[f]$ satisfying the conditions in Notation \ref{not:good-map}. When $n + \sum_{i}k_i \leq 3$, the log map $[f]$ flows into a unique (possibly stacky) point under the $\GG_m$-action. When $n+ \sum_i k_i \geq 4$, the log map $[f]$ flows into a point in the connected locus $\cF(B) \subset \fM_{\Gamma}(\X)$ by Proposition \ref{prop:attractor}. This proves the irreducibility. 
\end{proof}

\subsection{The unirationality}

Denote by
$\fM_{\lambda,\infty} \subset \fM_{\Gamma}(\X)$
the locally closed substack with underlying stable maps given by Proposition \ref{prop:limit}. Then $\fM_{\lambda,\infty}$ is in the torus fixed locus. Consider the open substack
\[
\fM_{\lambda} := \{f \in \fM_{\Gamma}(\X) \ | \ \lim_{t\mapsto \infty}t[f] \in \fM_{\lambda,\infty}\}.
\]
Then $\fM_{\lambda}$ is an irreducible log smooth stack. This means that the underlying stack of $\fM_{\lambda}$ has possibly toric singularities. To analyze the toric singularities, we first replace $\fM_{\lambda}$ by an \'etale cover to remove the monodromy of the log structure along the fixed locus $\fM_{\lambda,\infty}$ as follows.

Consider the discrete data $\Gamma'$ obtained by adding extra markings 
\begin{equation}\label{equ:rigidify-marking}
x_{1,1},\cdots,x_{1,k_1}; \cdots ; x_{l,1}, \cdots, x_{l,k_{l}}
\end{equation}
to $\Gamma$ with trivial contact orders. Denote by $\fM_{\Gamma'}(\X)_{0} \subset \fM_{\Gamma'}(\X)$ the dense open substack of automorphism free locus. We consider the locally closed substack
\begin{equation}\label{equ:rigidify-stack}
\fM' \subset \fM_{\Gamma'}(\X)_{0}
\end{equation}
such that any maps $f$ in $\fM'$ intersects $D_i$ transversally at the markings 
\[
x_{i,1},\cdots,x_{i,k_i}.
\]
We observe that $\fM'$ is log smooth, and the stack quotient $[\fM'/\fS]$ is birational to $\fM_{\lambda}$, where the symmetry group $\fS$ acts by permuting the markings (\ref{equ:rigidify-marking}) as in (\ref{equ:permuting-marking}). 

Consider the torus action on $\fM'$ induced by $\lambda$. Denote by 
\[\fM'_{\lambda,\infty} \subset \fM'\]
the closed substack consisting of underlying stable maps as in  Proposition \ref{prop:limit} with the extra markings (\ref{equ:rigidify-marking}) given by the intersection points with $D_i$ for all $i$. It is not hard to check that $\fM'_{\lambda,\infty}$ is an open substack of a component of the fixed loci of $\fM'$ under the torus action $\lambda$. A similar proof as in the case of $\fM_{\lambda,\infty}$ shows that

\begin{lemma}\label{lem:rigidified-fix-locus}
When (\ref{equ:enough-point}) holds, the underlying structure of $\fM'_{\lambda,\infty}$ is given by $M$ as in (\ref{equ:rigidified-handle-moduli}). Otherwise, $\fM'_{\lambda,\infty}$ is a single non-stacky point.
\end{lemma}

Note that we have an \'etale strict morphism
\begin{equation}\label{equ:forget-marking}
\fM'_{\lambda,\infty} \to \fM_{\lambda,\infty}
\end{equation}
by forgetting the extra marking (\ref{equ:rigidify-marking}). Furthermore, this morphism removes the monodromy of the log structure as follows:

\begin{lemma}\label{lem:zar-log}
The minimal log structure on $\fM'_{\lambda,\infty}$ is defined over Zariski site.
\end{lemma}
\begin{proof}
Since any log structure on a geometric point is Zariski, it suffices to prove the statement under the assumption (\ref{equ:enough-point}). For simplicity, we write $S = \fM'_{\lambda,\infty}$.

By Lemma \ref{lem:min-monoid} and the strictness of (\ref{equ:forget-marking}), the sheaf of groups $\ocM_{S}^{gp}$ is a locally constant sheaf on $S$. To proof the statement, it suffices to verify that $\ocM_{S}^{gp}$, hence $\ocM_{S}$ is globally constant. Denote by $f_{S}: \cC_{S}/S \to \X$ the universal family of stable log maps over $S$. Let $\sigma \subset \cC_S$ be the section over $S$ given by a fixed special point on the handle of each fiber. Note that $\sigma$ could be either node or marking. We next assume $\sigma$ is a node. The case $\sigma$ is a marking can be proved similarly.

Let $\cC_S^{\sharp} \to S^{\sharp}$ the canonical log structure associated to the underlying family of $\cC_S \to S$. By Lemma \ref{lem:rigidified-fix-locus}, both $\cM_{\cC_S^{\sharp}}$ and $\cM_{S^{\sharp}}$ are Zariski. In fact, $\ocM_{S^{\sharp}} = \sum_{e}\NN_e$ is a globally constant sheaf given by the product of the canonical log structure smoothing each node of the underlying family $\underline{\cC}_S \to \uS$ by \cite{LogSS}, Hence $\ocM_{\cC_S^{\sharp}}|_{\sigma}$ is a globally constant sheaf of monoids of the form:
\[
\ocM_{\cC_S^{\sharp}}|_{\sigma} = \sum_{e \neq \sigma}\NN_e \oplus \NN^2.
\]
Then we have
\begin{equation}\label{equ:curve-monoid}
\ocM_{\cC_S}|_{\sigma} = \ocM_{S}\oplus_{\NN_{\sigma}}\NN^2.
\end{equation}
Now the morphism of sheaves of groups:
\[
(\bar{f}_{S}^{\flat})^{gp}|_{\sigma}: f_S^*\ocM_{\X}^{gp}|_{\sigma} \to \ocM_{\cC_S}^{gp}|_{\sigma}
\]
is given by $(\bar{f}_{S}^{\flat})^{gp}|_{\sigma}(\delta) = a_{\delta} + b_{\delta}$ where $a_{\delta} \in \ocM_{S}$ and $b_{\delta} \in \NN^2$ such that $b_{\delta}$ is a globally constant section of the form $(b,0)$ or $(0,b)$ in $\NN^2$. Consider the composition
\begin{equation}\label{equ:projection}
\xymatrix{
f_S^*\ocM_{\X}^{gp}|_{\sigma}\oplus \ocM_{S^{\sharp}}^{gp} \ar[r] & \ocM_{\cC_S}^{gp}|_{\sigma}  \ar[r] & \ocM_{S}^{gp}
}
\end{equation}
where the second arrow is given by the projection induced by $a_{\delta} + b_{\delta} \mapsto a_{\delta}$. Using Lemma \ref{lem:min-monoid}(2) and (\ref{equ:edge-relation}), we verify that (\ref{equ:projection}) is an isomorphism over each fiber, hence is an isomorphism of sheaves of groups. On the other hand, since both $\ocM_{\X}^{gp}$ and $\ocM_{S^{\sharp}}^{gp}$ are globally constant, this implies that $\ocM_{S}^{gp}$ is also globally constant. This finishes the proof.
\end{proof}

Choose the open substack of $\fM'$:
\[
\fM'_{\lambda} := \{f \in \fM' \ | \ \lim_{t\mapsto \infty}t[f] \in \fM'_{\lambda,\infty}\}.
\]
Since $\fM'_{\lambda}$ is an irreducible, log smooth variety, by \cite{Ni, log-bound}, we may take the projective resolution
\begin{equation}\label{equ:resolution}
\phi: \fM^{res} \to \fM'_{\lambda}
\end{equation}
by a sequence of log \'etale blow-ups. 

\begin{lemma}
The resolution $\phi$ can be chosen to be $\Gm$-equivariant.
\end{lemma}
\begin{proof}
This can be seen from the construction of \cite{log-bound}. In fact, we may construct $\phi$ by first taking a barycentric subdivision as in \cite[Section 4.3]{log-bound}, which is $\Gm$-equivariant since the log structure of $\fM'_{\lambda}$ is stable under the torus action. Then by \cite[Lemma 4.4.1]{log-bound}, we may further construct $\phi$ by resolving the toric singularities Zariski locally over the barycentric subdivision as above, which is again $\Gm$-equivariant. 
\end{proof}

Let $\fM^{res}_{\infty} \subset \fM^{res}$ be the irreducible subvariety consisting of the $\Gm$-limits of general points in $\fM^{res}$ as $t \mapsto \infty$. Then $\phi (\fM^{res}_{\infty}) \subset \fM_{\lambda,\infty}$. By \cite{BB}, $\fM^{res}$ hence $\fM'$ is birational to a vector bundle over $\fM^{res}_{\infty}$. 

\begin{lemma}\label{lem:fix-loci-rational}
$\fM^{res}_{\infty}$ is rational.
\end{lemma}
\begin{proof}
Consider any reduced irreducible stratum $S$ of 
\[
\fM^{res}\times_{\fM'_{\lambda}}\fM'_{\lambda,\infty}
\]
with $\ocM_{S}$ a locally constant sheaf of monoids on $S$. Observe that $S$ is stable under the torus action. In fact, consider the projection
\[
\pi: S \to \fM'_{\lambda,\infty}.
\]
Then for any point $y \in \fM'_{\lambda,\infty}$, the fiber $\pi^{-1}(y)$ is a toric variety with its toric boundary given by the strata defined by the log structure on $S$. The torus action given by $\lambda$ on $\pi^{-1}(y)$ is compatible with the torus action of the toric variety, since the resolution (\ref{equ:resolution}) is locally given by toric blow-ups. Thus $\fM^{res}_{\infty}$ is given by some stratum $S$ as above. 

Consider the projection $\varphi: \fM^{res}_{\infty} \to \fM'_{\lambda,\infty}$. By further restricting to a Zariski open set of $\fM'_{\lambda,\infty}$, we may assume $\varphi$ is smooth. We next verify that $\varphi$ defines a Zariski locally trivial family of toric varieties.

Let $\cA_1$ and $\cA_2$ be the Artin fan of $\fM'_{\lambda}$ and $\fM'_{\lambda,\infty}$ respectively \cite[Section 3.2]{log-bound}. Since Artin fans in the initial factorization \cite[Proposition 3.1.1]{log-bound}, we have the following commutative diagram:
\[
\xymatrix{
\fM'_{\lambda,\infty} \ar[r] \ar[d] & \fM'_{\lambda} \ar[d] \\
\cA_2 \ar[r] & \cA_1.
}
\]
The resolution (\ref{equ:resolution}) is obtained via the subdivision $\cA_1' \to \cA_1$ as in \cite[Section 3.17]{log-bound}. This induces the subdivision $\cA_2':=\cA_1'\times_{\cA_1}\cA_2 \to \cA_2$ by \cite[Proposition 3.16]{log-bound}.  Since the log structure on $\fM'_{\lambda,\infty}$ is Zariski by Lemma \ref{lem:zar-log}, the Artin fan is of the form
\[
\cA_2 = [\spec \CC[N]/\spec \CC[N^{gp}]]
\]
where $N$ is the characteristic monoid over $\fM'_{\lambda,\infty}$. Thus the subdvision $\cA_2'$ is obtained by a sequence of toric blow-ups of $\cA_2$. Since 
\[
\fM^{res}\times_{\fM'_{\lambda}}\fM'_{\lambda,\infty} = \cA'_{2}\times_{\cA_2}\fM'_{\lambda,\infty},
\]
by Lemma \ref{lem:zar-log} we find $\phi$ is a family of toric varieties which admits a Zariski local trivialization. Finally, the statement follows from the rationality of $M$ and Lemma \ref{lem:rigidified-fix-locus}.
\end{proof}

Summing up the above argument, we have

\begin{proposition}\label{prop:quotient-of-rational}
Under Assumption \ref{assu:good-curve} for the curve class $\beta$, the stack $\fM_{\Gamma}(\X)$ is birational to a quotient of the rational variety $\fM'$ by the product of symmetric groups $\fS$ as in (\ref{equ:permuting-marking}).
\end{proposition}

This concludes the proof of unirationality in Theorem \ref{thm:main}. \qed

\begin{corollary}\label{cor:rationality}
Assume that $k_i \leq 1$ for all $i$ in (\ref{equ:curve-class-ad}), and Assumption \ref{assu:good-curve} holds. Then $\fM_{\Gamma}(\X)$ is rational.
\end{corollary}
\begin{proof}
Under the assumption of the statement, the group $\fS$ is trivial.
\end{proof}

\bibliographystyle{amsalpha}             
\bibliography{Rationality}

\providecommand{\bysame}{\leavevmode\hbox to3em{\hrulefill}\thinspace}
\providecommand{\MR}{\relax\ifhmode\unskip\space\fi MR }
\providecommand{\MRhref}[2]{%
  \href{http://www.ams.org/mathscinet-getitem?mr=#1}{#2}
}
\providecommand{\href}[2]{#2}
\begin{thebibliography}{ACMW14}

\bibitem[AC14]{AC}
Dan Abramovich and Qile Chen, \emph{Stable logarithmic maps to
  {D}eligne-{F}altings pairs {II}}, Asian Journal of Mathematics \textbf{18}
  (2014), no.~3, 465--488, arXiv:1102.4531.

\bibitem[ACGS]{ACGS}
Dan Abramovich, Qile Chen, Mark Gross, and Bernd Siebert, in progress.

\bibitem[ACMW14]{log-bound}
Dan Abrmovich, Qile Chen, Steffen Marcus, and Jonathan Wise, \emph{Boundedness
  of the space of stable log maps}, arXiv:1408.0869.

\bibitem[Ahi83]{Ahiezer}
Dmitry Ahiezer, \emph{Equivariant completions of homogeneous algebraic
  varieties by homogeneous divisors}, Ann. Global Anal. Geom. \textbf{1}
  (1983), no.~1, 49--78. \MR{739893 (85j:32052)}

\bibitem[BB73]{BB}
A.~Bia{\l}ynicki-Birula, \emph{Some theorems on actions of algebraic groups},
  Ann. of Math. (2) \textbf{98} (1973), 480--497. \MR{0366940 (51 \#3186)}

\bibitem[Bog86]{Bog}
F.~A. Bogomolov, \emph{Rationality of the moduli of hyperelliptic curves of
  arbitrary genus}, Proceedings of the 1984 {V}ancouver conference in algebraic
  geometry, CMS Conf. Proc., vol.~6, Amer. Math. Soc., Providence, RI, 1986,
  pp.~17--37. \MR{846014 (87k:14022)}

\bibitem[Bri89]{Brion-rankone}
Michel Brion, \emph{On spherical varieties of rank one (after {D}. {A}hiezer,
  {A}. {H}uckleberry, {D}. {S}now)}, Group actions and invariant theory
  ({M}ontreal, {PQ}, 1988), CMS Conf. Proc., vol.~10, Amer. Math. Soc.,
  Providence, RI, 1989, pp.~31--41. \MR{1021273 (91a:14028)}

\bibitem[Bri03]{Brion-onecycle}
\bysame, \emph{The cone of effective one-cycles of certain {$G$}-varieties}, A
  tribute to {C}. {S}. {S}eshadri ({C}hennai, 2002), Trends Math.,
  Birkh\"auser, Basel, 2003, pp.~180--198. \MR{2017584 (2004m:14008)}

\bibitem[Bri07]{Brion-evb}
\bysame, \emph{Construction of equivariant vector bundles}, Algebraic groups
  and homogeneous spaces, Tata Inst. Fund. Res. Stud. Math., Tata Inst. Fund.
  Res., Mumbai, 2007, pp.~83--111. \MR{2348903 (2008k:14101)}

\bibitem[Bri12]{Brion-spherical}
\bysame, \emph{Spherical varieties}, Highlights in {L}ie algebraic methods,
  Progr. Math., vol. 295, Birkh\"auser/Springer, New York, 2012, pp.~3--24.
  \MR{2866845 (2012j:14072)}

\bibitem[Che14]{Chen}
Qile Chen, \emph{Stable logarithmic maps to {D}eligne-{F}altings pairs {I}},
  Ann. of Math. (2) \textbf{180} (2014), no.~2, 455--521, arXiv:1008.3090.
  \MR{3224717}

\bibitem[CS13]{CM}
Qile Chen and Matthew Satriano, \emph{Chow quotients of toric varieties as
  moduli of stable log maps}, Algebra Number Theory \textbf{7} (2013), no.~9,
  2313--2329. \MR{3152015}

\bibitem[CZ14]{A1}
Qile Chen and Yi~Zhu, \emph{$\mathbb{A}^1$-curves on log smooth varieties},
  arXiv:1407.5476.

\bibitem[GS13]{GS}
Mark Gross and Bernd Siebert, \emph{Logarithmic {G}romov-{W}itten invariants},
  J. Amer. Math. Soc. \textbf{26} (2013), no.~2, 451--510. \MR{3011419}

\bibitem[Har77]{Hartshorne}
Robin Hartshorne, \emph{Algebraic geometry}, Springer-Verlag, New York, 1977,
  Graduate Texts in Mathematics, No. 52. \MR{0463157 (57 \#3116)}

\bibitem[Kat84]{Ka}
P.~I. Katsylo, \emph{Rationality of fields of invariants of reducible
  representations of the group {${\rm SL}_2$}}, Vestnik Moskov. Univ. Ser. I
  Mat. Mekh. (1984), no.~5, 77--79. \MR{764040 (86c:14009)}

\bibitem[Kat89]{KKato}
Kazuya Kato, \emph{Logarithmic structures of {F}ontaine-{I}llusie}, Algebraic
  analysis, geometry, and number theory ({B}altimore, {MD}, 1988), Johns
  Hopkins Univ. Press, Baltimore, MD, 1989, pp.~191--224. \MR{MR1463703
  (99b:14020)}

\bibitem[Kat96]{FKato}
Fumiharu Kato, \emph{Log smooth deformation theory}, Tohoku Math. J. (2)
  \textbf{48} (1996), no.~3, 317--354. \MR{1404507 (99a:14012)}

\bibitem[Kno91]{Knop}
Friedrich Knop, \emph{The {L}una-{V}ust theory of spherical embeddings},
  Proceedings of the {H}yderabad {C}onference on {A}lgebraic {G}roups
  ({H}yderabad, 1989), Manoj Prakashan, Madras, 1991, pp.~225--249. \MR{1131314
  (92m:14065)}

\bibitem[Kno96]{Knop96}
\bysame, \emph{Automorphisms, root systems, and compactifications of
  homogeneous varieties}, J. Amer. Math. Soc. \textbf{9} (1996), no.~1,
  153--174. \MR{1311823 (96c:14037)}

\bibitem[KP01]{KP}
B.~Kim and R.~Pandharipande, \emph{The connectedness of the moduli space of
  maps to homogeneous spaces}, Symplectic geometry and mirror symmetry
  ({S}eoul, 2000), World Sci. Publ., River Edge, NJ, 2001, pp.~187--201.
  \MR{1882330 (2002k:14021)}

\bibitem[Lun01]{Luna}
D.~Luna, \emph{Vari\'et\'es sph\'eriques de type {$A$}}, Publ. Math. Inst.
  Hautes \'Etudes Sci. (2001), no.~94, 161--226. \MR{1896179 (2003f:14056)}

\bibitem[Niz06]{Ni}
Wies{\l}awa Nizio{\l}, \emph{Toric singularities: log-blow-ups and global
  resolutions}, J. Algebraic Geom. \textbf{15} (2006), no.~1, 1--29.
  \MR{2177194 (2006i:14015)}

\bibitem[Ols03a]{LogStack}
Martin~C. Olsson, \emph{Logarithmic geometry and algebraic stacks}, Ann. Sci.
  \'Ecole Norm. Sup. (4) \textbf{36} (2003), no.~5, 747--791. \MR{MR2032986
  (2004k:14018)}

\bibitem[Ols03b]{LogSS}
\bysame, \emph{Universal log structures on semi-stable varieties}, Tohoku Math.
  J. (2) \textbf{55} (2003), no.~3, 397--438. \MR{MR1993863 (2004f:14025)}

\bibitem[Ols07]{LogCurve}
\bysame, \emph{({L}og) twisted curves}, Compos. Math. \textbf{143} (2007),
  no.~2, 476--494. \MR{MR2309994 (2008d:14021)}

\bibitem[Res10]{min-rank}
N.~Ressayre, \emph{Spherical homogeneous spaces of minimal rank}, Adv. Math.
  \textbf{224} (2010), no.~5, 1784--1800. \MR{2646110 (2011h:14071)}

\bibitem[Tho98]{Thomsen}
Jesper~Funch Thomsen, \emph{Irreducibility of
  {$\overline{M}_{0,n}(G/P,\beta)$}}, Internat. J. Math. \textbf{9} (1998),
  no.~3, 367--376. \MR{1625369 (99g:14032)}

\bibitem[Tim11]{Timashev}
Dmitry~A. Timashev, \emph{Homogeneous spaces and equivariant embeddings},
  Encyclopaedia of Mathematical Sciences, vol. 138, Springer, Heidelberg, 2011,
  Invariant Theory and Algebraic Transformation Groups, 8. \MR{2797018
  (2012e:14100)}

\bibitem[Wis]{wise}
Jonathan Wise, \emph{Moduli of morphisms of logarithmic schemes},
  arXiv:1408.0037.

\end{thebibliography}

\end{document}